\documentclass[12pt]{article}

\usepackage{authblk}
\usepackage{cite}
\usepackage{amsmath}
\usepackage{amsthm,amsfonts,amssymb}           
\usepackage{bbm}
\usepackage[left=2.5cm,
		right=2.5cm,
		top=2.5cm,
		bottom=4cm
		]{geometry}
\usepackage[pdfstartview=FitH,
		CJKbookmarks=true,
		bookmarksnumbered=true,
		bookmarksopen=true,
		colorlinks,
		linkcolor=blue,
		anchorcolor=blue,
		citecolor=blue,
		urlcolor=blue
		]{hyperref}
\numberwithin{equation}{section}

\newtheorem{thm}{Theorem}[section]
\newtheorem{defn}{Definition}[section]
\newtheorem{lem}[thm]{Lemma}
\newtheorem{prop}[thm]{Proposition}

\def\p{\psi}
\def\f{\phi}
\def\l{\lambda}
\def\sg{\sigma}
\def\z{\zeta}
\newcommand{\R}{\mathbb{R}}

\newcommand{\poly}{\mathcal{P}^n}					
\newcommand{\polyo}{\mathcal{P}^n_o}				
\newcommand{\polye}{\mathcal{P}^2}					
\newcommand{\polyeo}{\mathcal{P}^2_o}				
\newcommand{\GL}{\mathrm{GL}(n)}					
\newcommand{\SL}{\mathrm{SL}(n)}					
\newcommand{\SLe}{\mathrm{SL}(2)}					
\newcommand{\SLs}{\mathrm{SL}(3)}					
\newcommand{\norm}[1]{\left\Vert{#1}\right\Vert}		
\newcommand{\set}[1]{\left\{{#1}\right\}}			
\newcommand{\inp}[2]{#1\cdot#2}						
\newcommand{\conv}[1]{\left[#1\right]}				
\newcommand{\id}{{\bf1}}							
\newcommand{\bd}{\mathrm{bd}\,}						
\newcommand{\rint}{\mathrm{relint}\,}				
\newcommand{\intp}{\mathrm{int}\,}					
\newcommand{\aff}{\mathrm{aff}\,}					
\newcommand{\np}{\mathcal{N}}						
\newcommand{\npo}{\mathcal{N}_o}					
\newcommand{\spl}{\mathcal T^n_o}					
\newcommand{\splo}{\mathcal T^n_o\setminus\set{o}}	

\newcommand{\vo}[1]{B_{#1}}              			
\newcommand{\ve}{A}              					
\newcommand{\mv}{M^{1,0}}          				    
\newcommand{\fv}{M^{0,1}}							

\makeatletter
\newcommand{\subjclass}[2][1991]{%
  \let\@oldtitle\@title%
  \gdef\@title{\@oldtitle\footnotetext{#1 \emph{Mathematics subject classification.} #2}}%
}
\newcommand{\keywords}[1]{%
  \let\@@oldtitle\@title%
  \gdef\@title{\@@oldtitle\footnotetext{\emph{Key words and phrases.} #1.}}%
}
\makeatother

\title{\bf{SL($n$) contravariant vector valuations}}
\author[1]{Jin Li}
\author[2,3]{Dan Ma}
\author[4]{Wei Wang}
\affil[1]{Institut f\"{u}r Diskrete Mathematik und Geometrie, Technische Universit\"{a}t Wien, 1040 Vienna, Austria\\
	\href{mailto: Jin Li}{lijin2955@gmail.com}}
\affil[2]{Department of Mathematics, Shanghai Normal University, Shanghai 200234, China\\\href{mailto: Dan Ma}{madan@shnu.edu.cn}}
\affil[3]{Department of Mathematics, Karlsruhe Institute of Technology (KIT), D-76128 Karlsruhe, Germany}
\affil[4]{School of Mathematics and Computational Science, Hunan University of Science and Technology,
	Xiangtan 411201, China\\\href{mailto: Wei Wang}{wwang@hnust.edu.cn}}
\date{}

\begin{document}

\maketitle

\begin{abstract}
All SL($n$) contravariant vector valuations on polytopes in $\R^n$
are completely classified without any additional assumptions.
The facet vector is defined. It turns out to be the unique such valuation for $n\geq3$.
In dimension two, the classification corresponds to the case of SL(2) covariant valuations.
\end{abstract}


\section{Introduction}

The study of geometric notions which are
compatible with transformation groups are important tasks in geometry
as proposed in Felix Klein's Erlangen program in 1872.
As many functions defined on geometric objects satisfy
the inclusion-exclusion principle, the property of being a valuation
is natural to consider in the classification of those functions.
Here, a function $ Z $ defined on $\poly$, the space of all polytopes in $\R^n$,
and taking values in an abelian semigroup is called a {\em valuation} if
\begin{equation}\label{eqn:val}
	 Z (P)+ Z (Q)= Z (P\cup Q)+ Z (P\cap Q)
\end{equation}
for every $P,Q,P\cup Q\in\poly$. A function $ Z $ defined on
some subspace of $\poly$ is also called
a valuation if \eqref{eqn:val} holds
whenever $P,Q,P\cup Q,P\cap Q$ contained in this subspace.
Valuations also have their origins in Dehn's solution of
Hilbert's Third Problem in 1901.
The most famous result is Hadwiger's characterization theorem
which classifies all continuous and rigid motion invariant real valuations
on the space of convex bodies in $\R^n$.
This celebrated result initiated a systematic study on the classification
of valuations compatible with certain linear transforms.

These studies are also a classical part of geometry
with important applications in integral geometry
(see \cite[Chap. 7]{Gru07},\cite{KR97},\cite[Chap. 6]{Sch14}).
They turned out to be extremely fruitful and useful
especially in the affine geometry of convex bodies
(see \cite{Ale99,Ale01,HP14a,Kla96,LYL15,BF11}).
Examples are intrinsic volumes \cite{Lud02a,HP14b},
affine surface areas \cite{Lud10b,LR10},
the projection bodies \cite{Hab12b,Lud02b,Lud05,LL16},
the intersection bodies \cite{Lud06}
and other Minkowski valuations \cite{BL19,SW18,Sch10}.

The aim of this paper is to obtain a complete classification of $\SL$
contravariant vector valuations on polytopes without any additional assumptions.

A function $ Z :\poly\to\R^n$ is called vector valuations if the addition in \eqref{eqn:val} is the vector addition.
It is called \emph{$\SL$ contravariant} if $ Z (\f P)=\f^{-t} Z (P)$ for all $P\in\poly$ and $\f\in\SL$,
and is called \emph{$\SL$ covariant} if $ Z (\f P)=\f Z (P)$ for all $P\in\poly$ and $\f\in\SL$.
If $Z$ is either $\SL$ contravariant or $\SL$ covariant, then $Z$ is \emph{$\SL$ intertwining}.
In 2002, Ludwig \cite{Lud02c} established the first classification
of measurable, $\SL$ intertwining vector valuations on $\mathcal{P}^n_{(o)}$ with some assumptions of homogeneity,
where $\mathcal{P}^n_{(o)}$ is the space of polytopes in $\R^n$ that contain the origin in their interiors.
Later, Haberl and Parapatits \cite{HP16} removed the homogeneity assumption in Ludwig's result.
Recently, Zeng and the second author \cite{ZM18} obtained a complete classification
of $\SL$ covariant vector valuations on $\poly$ without any additional assumptions.
There are also some interesting characterizations of matrix and tensor valuations
(see \cite{Lud03,HP17,Ma19,MW19,ABDK19}).
Surprisingly, classifications of $\SL$ contravariant vector valuations are still missing on any spaces and with any conditions.

An intuitive example of $\SL$ contravariant vector valuation
is the sum of all facet normals. However, the Minkowski relation
shows that it vanishes. More precisely, in the discrete case,
it means the following. For $P\in\poly$, we have
\[\sum_{u\in\np(P)}a_P(u)u=o,\]
where $\np(P)$ denotes the set of all outer unit normals of facets of $P$,
$a_P(u)$ denotes the ($n-1$)-dimensional volume of $F(P,u)$,
and $F(P,u)$ denotes the facet of $P$ with outer unit normal $u$ (see \cite{Sch14}).
Nevertheless, for polytopes containing the origin, taking the partial sum
over facets that do not contain the origin, we get a non-zero valuation.
For a solution of Cauchy's functional equation $\z:[0,\infty)\to\R$,
the facet vector $\fv_\z(P)$ of $P\in\poly$ is defined by
\[\fv_\z(P)=\sum_{u\in\np(P)\setminus\npo(P)}\frac{\z(V(P,u))}{|h_P(u)|}u,\]
where $\npo(P)$ denotes the set of outer unit normals of facets of $P$ that contain the origin in their affine hulls,
$V(P,u)$ denotes the volume of the cone $\conv{o,F(P,u)}$,
the convex hull of $F(P,u)$ and the origin,
and $h_P(u)=\max\set{\inp{x}{u}:x\in P}$ denotes the support function of $P$.
We use the notation $\fv$ coinciding with $(0,1)$-tensor in \cite{HP17}.
Also, it is related with $(0,1)$-Minkowsi tensor in \cite[\S 5.4.2]{Sch14}.

In this paper, we show that the facet vector
is essentially the unique $\SL$ contravariant vector valuation on $\polyo$
for $n\geq3$.

Let $\polyo$ be the space of polytopes in $\R^n$ that contain the origin,

\begin{thm}\label{thm:mainN0}
	Let $n\geq3$.
	A function $ Z :\polyo\to\R^n$ is an $\SL$ contravariant valuation
	if and only if there exists a solution of Cauchy's functional equation
	$\z:[0,\infty)\to\R$ such that
	\[ Z (P)=\fv_\z(P)\]
	for every $P\in\polyo$.
\end{thm}

Using a relation with $\SLe$ covariant vector valuations,
we obtain the classification in the case of dimension two.
We also find that the vector $\vo{\z}$ defined in \cite{ZM18} turns out to
be a rotation of the facet vector in this case (see Section 3 for details).
\begin{thm}\label{thm:main20}
	A function $ Z :\polyeo\to\R^2$ is an $\SLe$ contravariant valuation
	if and only if there exist constants $c_1, c_2\in\R$
	and a solution of Cauchy's functional equation $\z:[0,\infty)\to\R$ such that
	\[ Z (P)=\fv_\z(P)+c_1\rho_{\frac{\pi}{2}}\mv(P)+c_2\rho_{\frac{\pi}{2}}\ve(P)\]
	for every $P\in\polyeo$, where
	$\rho_{\frac{\pi}{2}}$ is the counter-clockwise rotation in $\R^2$ of the angle $\frac{\pi}{2}$.
\end{thm}
Here, for $P\in\poly$, $\mv(P)$ is the \emph{moment vector} of $P$,
which is defined by $\mv(P)=\int_Pxdx$.
The notation also coincides with $(1,0)$-tensor in \cite{HP17} and is related with $(1,0)$-Minkowsi tensor in \cite[\S 5.4.2]{Sch14}.
The valuation $\ve:\polyeo\to\R^2$ is defined by $\ve(P)=v+w$
if $\dim P=2$ and $P$ has two edges $\conv{o,v}$ and $\conv{o,w}$,
or $\dim P=2$ and $P$ has an edge $\conv{v,w}$ that contains the origin in its relative interior;
$\ve(P)=2(v+w)$ if $\dim P=1$ and $P=\conv{v,w}$ contains the origin;
$\ve(P)=0$ otherwise.

Similar to the classification of convex body valuations by Schuster and Wannerer \cite{SW12},
we further extend these results to $\poly$.

\begin{thm}\label{thm:mainN}
	Let $n\geq3$.
	A function $ Z :\poly\to\R^n$ is an $\SL$ contravariant valuation
	if and only if there exist solutions of Cauchy's functional equation
	$\z_1,\z_2:[0,\infty)\to\R$ such that
	\begin{equation}\label{eqn:mainN}
		 Z (P)=\fv_{\z_1}(P)+\fv_{\z_2}(\conv{o,P})
	\end{equation}
	for every $P\in\poly$, where $\conv{o,P}$ is the convex hull of $P$ and the origin.
\end{thm}

Again, the case of dimension two is different.

\begin{thm}\label{thm:main2}
	A function $ Z :\polye\to\R^2$ is an $\SLe$ contravariant valuation
	if and only if there exist constants $c_1,c_2,\tilde c_1,\tilde c_2\in\R$
	and solutions of Cauchy's functional equation $\z_1,\z_2:[0,\infty)\to\R$ such that
	\begin{align*}
Z (P)&=\fv_{\z_1}(P)+\fv_{\z_2}(\conv{o,P})+c_1\rho_{\frac{\pi}{2}}\mv(P)+\tilde c_1\rho_{\frac{\pi}{2}}\mv(\conv{o,P}) \\
	&\qquad +c_2\rho_{\frac{\pi}{2}}\ve(\conv{o,P})+\tilde c_2\rho_{\frac{\pi}{2}}\ve(\conv{o,v_1,\ldots,v_r})
\end{align*}
	for every polytope $P\in\polye$ with vertices $v_1,\ldots,v_r$ visible from the origin and labeled counter-clockwisely,
	where a vertex $v$ of $P$ is called \emph{visible} from the origin if $P\cap\rint\conv{o,v}=\varnothing$.
\end{thm}

It should be remarked that vector valuations are special Minkowski valuations \cite{Hab12b},
since vectors can be viewed as convex bodies and the vector addition coincides with the Minkowski addition.
Also, vectors can be viewed as linear functions on $\R^n$.
Hence vector valuations are also embedded in the space of continuous-function valued valuations \cite{Li18}.
However, classifications of valuations in \cite{Hab12b,Li18} both need some assumptions of regularity.
But as we have seen, it is not a problem for vector valuations.


\section{Notation and preliminaries}

We work in $n$-dimensional Euclidean space $\R^n$
with the standard basis $\set{e_1,\ldots,e_n}$.
We write a vector $x\in\R^n$ in coordinates by $x=(x_1,\ldots,x_n)^t$.
The inner product of $x,y\in\R^n$ is denoted by $\inp{x}{y}$.
Denote the vector with all coordinates 1 by $\id$,
the $n\times n$ identity matrix by $I_n=(e_1,\ldots,e_n)$
and the determinant of a matrix $A$ by $\det A$.
The affine hull, the boundary, the dimension, the interior and the relative interior
of a given set in $\R^n$ are denoted by aff, bd, dim, int and relint, respectively.

Denote by $\conv{v_1,\ldots,v_k}$ the convex hull of $v_1,\ldots,v_k\in\R^n$.
A polytope is the convex hull of finitely many points in $\R^n$.
Two basic classes of polytopes are the $k$-dimensional standard simplex
$T^k=\conv{o,e_1,\ldots,e_k}$ and one of their ($k-1$)-dimensional facets
$\tilde T^k=\conv{e_1,\ldots,e_k}$.
For $i=1,\ldots,n$, let $\mathcal T^i$ denote the set of $i$-dimensional simplices
with one vertex at the origin, and $\tilde{\mathcal T}^i$ denote
the set of $(i-1)$-dimensional simplices $T\subset\R^n$ with $o\notin\aff T$.
Indeed, every polytope can be triangulated into simplices.
We define a \emph{triangulation} of a $k$-dimensional polytope $P$ into simplices
as a set of $k$-dimensional simplices $\set{T_1, \ldots, T_r}$
which have pairwise disjoint interiors, with $P=\cup T_i$
and with the property that for arbitrary $1\leq i_1<\cdots <i_j\leq r$ the intersections
$T_{i_1}\cap\cdots\cap T_{i_j}$ are again simplices.

We refer to \cite[Chap. 7]{Gru07}, \cite{KR97} and \cite[Chap. 6]{Sch14}
for classical backgroud on valuations. Let $\mathcal{Q}^n$ be either $\polyo$ or $\poly$.
First, we have the inclusion-exclusion principle (see \cite{KR97}).

\begin{lem}
	Let $ Z :\mathcal{Q}^n\to\R^n$ be a valuation. Then
	\[ Z (P_1\cup\cdots\cup P_k)=\sum_{\varnothing\neq S\subseteq\set{1, 2,\ldots, k}}(-1)^{|S|-1} Z (\bigcap\limits_{i\in S}P_i)\]
	for all $k\in \mathbb N$ and $P_1, P_2, \ldots, P_k\in\mathcal{Q}^n$ with $P_1\cup\cdots\cup P_k\in\mathcal{Q}^n$.
\end{lem}

We can use triangulations and the inclusion-exclusion principle to get the following result (see e.g., \cite[Lemma 4.5 and Lemma 4.6]{Li18}).
	
\begin{lem}\label{lemuq}
	Let $Z$ and $Z'$ be $\SL$ contravariant vector valuations on $\polyo$. If $Z (sT^d) = Z' (sT^d)$ for every $s>0$ and $0 \leq d \leq n$, then $Z P = Z' P$ for every $P \in \polyo$.
\end{lem}
	
\begin{lem}\label{lemuq2}
	Let $Z$ and $Z'$ be $\SL$ contravariant vector valuations  on $\poly$.
	If $Z (sT^d) = Z' (sT^d)$ and $Z(s\tilde T^d) = Z'(s\tilde T^d)$
	for every $s>0$ and $0 \leq d \leq n$, then $Z P = Z' P$ for every $P \in \poly$.
\end{lem}

A valuation on $\mathcal{Q}^n$ is called {\em simple} if it vanishes
on every lower dimensional $P\in\mathcal{Q}^n$.

Next, we mention a series of triangulations that will be used several times in this paper.
Let $\l\in(0,1)$ and denote by $H$ the hyperplane through the origin with
the normal vector $(1-\l)e_1-\l e_2$. Write
\[H^+=\{x\in\R^n:\inp{x}{((1-\l)e_1-\l e_2)}\geq 0\}\text{ and }
H^-=\{x\in\R^n:\inp{x}{((1-\l)e_1-\l e_2)}\leq 0\}.\]
Clearly, $H^+$ and $H^-$ are the two halfspaces bounded by $H$.
This hyperplane induces the series of triangulations of $T^i$ as well as $\tilde T^i$
for $i=2,\ldots,n$. There are two representations corresponding to these triangulations
due to the following definitions.

Let $\hat{T}^{k-1}=[o,e_1,e_3,\dots,e_k]$ for $2 \leq k \leq n$.
\begin{defn}\label{defn:fp1}
	For $\l\in(0,1)$, define the linear transform $\f_1\in\SL$ by
	\[\f_1e_1=\l e_1+(1-\l)e_2,\;\f_1e_2=e_2,\;\f_1e_n=e_n/\l,\;\f_1e_j=e_j,\;\text{ where }j\neq 1,2,n,\]
	and $\p_1\in\SL$ by
	\[\p_1e_1=e_1,\;\p_1e_2=\l e_1+(1-\l)e_2,\;\p_1e_n=e_n/(1-\l),\;\p_1e_j=e_j,\;\text{ where }j\neq 1,2,n.\]
\end{defn}
Let $i<n$.
It is clear that $T^i\cap H^+=\p_1T^i$, $T^i\cap H^-=\f_1T^i$ and $T^i\cap H=\f_1\hat{T}^{i-1}$.
Let $ Z :\polyo\to\R^n$ be an $\SL$ contravariant valuation.
By the inclusion-exclusion principle, we have
\[ Z (T^i)+ Z (T^i\cap H)= Z (T^i\cap H^+)+ Z (T^i\cap H^-).\]
Thus,
\[ Z (T^i)+ Z (\f_1\hat{T}^{i-1})= Z (\f_1T^i)+ Z (\p_1T^i).\]
Since $ Z $ is $\SL$ contravariant, we derive
\begin{equation}\label{eqn:extk}
	\left(\f_1^{-t}+\p_1^{-t}-I_n\right) Z (T^i)=\f_1^{-t} Z (\hat{T}^{i-1}).
\end{equation}

\begin{defn}\label{defn:fp2}
	For $\l\in(0,1)$, define the linear transform $\f_2\in\GL$ by
	\[\f_2e_1=\l e_1+(1-\l)e_2,\quad\f_2e_2=e_2,\quad\f_2e_j=e_j,\quad\text{ where }j=3,\ldots,n,\]
	and $\p_2\in\GL$ by
	\[\p_2e_1=e_1,\quad\p_2e_2=\l e_1+(1-\l)e_2,\quad\p_2e_j=e_j,\quad\text{ where }j=3,\ldots,n.\]
\end{defn}
Now, we consider $sT^n$ for $s>0$.
It is clear that $sT^n\cap H^+=\p_2sT^n$, $sT^n\cap H^-=\f_2sT^n$ and $sT^n\cap H=\f_2s\hat{T}^{n-1}$.
Let $ Z :\polyo\to\R^n$ be an $\SL$ contravariant valuation.
Again, by the inclusion-exclusion principle, we have
\[ Z (sT^n)+ Z (sT^n\cap H)= Z (sT^n\cap H^+)+ Z (sT^n\cap H^-).\]
Thus,
\[ Z (sT^n)+ Z (\f_2s\hat{T}^{n-1})= Z (\f_2sT^n)+ Z (\p_2sT^n).\]
Since $\f_2/\l^{\frac{1}{n}}$ and $\p_2/(1-\l)^{\frac{1}{n}}$ belong to $\SL$, we obtain
\[ Z (sT^n)+\l^{\frac{1}{n}}\f_2^{-t} Z (\l^{\frac{1}{n}}s\hat{T}^{n-1})=
\l^{\frac{1}{n}}\f_2^{-t} Z (\l^{\frac{1}{n}}sT^n)+(1-\l)^{\frac{1}{n}}\p_2^{-t} Z ((1-\l)^{\frac{1}{n}}sT^n).\]
Replacing $s$ by $s^{\frac{1}{n}}$ in the equation above yields
\begin{equation}\label{eqn:extn}
\begin{split}
	& Z (s^{\frac{1}{n}}T^n)+\l^{\frac{1}{n}}\f_2^{-t} Z ((\l s)^{\frac{1}{n}}\hat{T}^{n-1})\\
	=&\l^{\frac{1}{n}}\f_2^{-t} Z ((\l s)^{\frac{1}{n}}T^n)+(1-\l)^{\frac{1}{n}}\p_2^{-t} Z (((1-\l)s)^{\frac{1}{n}}T^n).
\end{split}
\end{equation}


\section{The facet vector}

First, we show that the facet vector is a simple valuation on $\poly$.

\begin{lem}\label{thm:simpval}
	Let $\z:[0,\infty)\to\R$ be a solution of Cauchy's functional equation. Then,
	the facet vector $\fv_\z:\poly\to\R^n$ is a simple valuation.
\end{lem}
\begin{proof}
	In order to prove that $\fv_\z$ is a valuation, we need to show that
	\begin{equation}\label{eqn:3.1}
		\fv_\z(P\cup Q)+\fv_\z(P\cap Q)=\fv_\z(P)+\fv_\z(Q)
	\end{equation}
	for all $P,Q\in\poly$ with $P\cup Q\in\poly$. We distinguish three sets of unit vectors:
	
	\[I_1:=\{u\in S^{n-1}: h_P(u)<h_Q(u)\},\]
	\[I_2:=\{u\in S^{n-1}: h_P(u)=h_Q(u)\},\]
	\[I_3:=\{u\in S^{n-1}: h_P(u)>h_Q(u)\}.\]
	Note that the sets $I_1, I_3$ are both open and that $h_{P\cup Q}=\max\{h_{P}, h_{Q}\}$ and $h_{P\cap Q}=\min\{h_{P}, h_{Q}\}$
	if $P\cup Q$ is convex. Recall that $a_P(u)$ is the ($n-1$)-dimensional volume of $F(P,u)$. Then,
	\[V(P,u)=\frac{1}{n}a_P(u)h_P(u).\]
	For $u\in I_1$, we have
	\[a_{P\cup Q}(u)=a_{Q}(u),\,h_{P\cup Q}(u)=h_{Q}(u),\,a_{P\cap Q}(u)=a_{P}(u),\,h_{P\cap Q}(u)=h_{P}(u).\]
	Thus,
	\[V(P\cup Q,u)=V(Q,u)\text{ and }V(P\cap Q)=V(P,u),\quad\text{for }u\in I_1.\]
	Analogous for $I_3$.
	Note that
	
	\[(\np(P\cup Q)\setminus\npo(P\cup Q))\cap I_1=(\np(Q)\setminus\npo(Q))\cap I_1,\]
	\[(\np(P\cap Q)\setminus\npo(P\cap Q))\cap I_1=(\np(P)\setminus\npo(P))\cap I_1,\]
	\[(\np(P\cup Q)\setminus\npo(P\cup Q))\cap I_3=(\np(P)\setminus\npo(P))\cap I_3,\]
	\[(\np(P\cap Q)\setminus\npo(P\cap Q))\cap I_3=(\np(Q)\setminus\npo(Q))\cap I_3.\]
	Therefore, we have
	\begin{align*}
		&\displaystyle\sum_{u\in(\np(P\cup Q)\setminus\npo(P\cup Q))\cap I_1}\frac{\z(V(P\cup Q,u))}{h_{P\cup Q}(u)}u
		+\displaystyle\sum_{u\in(\np(P\cap Q)\setminus\npo(P\cap Q))\cap I_1}\frac{\z(V(P\cap Q,u))}{h_{P\cap Q}(u)}u\\
		&~+\displaystyle\sum_{u\in(\np(P\cup Q)\setminus\npo(P\cup Q))\cap I_{3}}\frac{\z(V(P\cup Q,u))}{h_{P\cup Q}(u)}u
		+\displaystyle\sum_{u\in(\np(P\cap Q)\setminus\npo(P\cap Q))\cap I_{3}}\frac{\z(V(P\cap Q,u))}{h_{P\cap Q}(u)}u\\
		&=\displaystyle\sum_{u\in(\np(Q)\setminus\npo(Q))\cap I_{1}}\frac{\z(V(Q,u))}{h_Q(u)}u		
		+\displaystyle\sum_{u\in(\np(P)\setminus\npo(P))\cap I_{1}}\frac{\z(V(P,u))}{h_P(u)}u\\
		&~+\displaystyle\sum_{u\in(\np(P)\setminus\npo(P))\cap I_{3}}\frac{\z(V(P,u))}{h_P(u)}u
		+\displaystyle\sum_{u\in(\np(Q)\setminus\npo(Q))\cap I_{3}}\frac{\z(V(Q,u))}{h_Q(u)}u.
	\end{align*}
	It follows that \eqref{eqn:3.1} is equivalent to
	\begin{align*}
		&\displaystyle\sum_{u\in(\np(P\cup Q)\setminus\npo(P\cup Q))\cap I_2}\frac{\z(V(P\cup Q,u))}{h_{P\cup Q}(u)}u
		+\displaystyle\sum_{u\in(\np(P\cap Q)\setminus\npo(P\cap Q))\cap I_2}\frac{\z(V(P\cap Q,u))}{h_{P\cap Q}(u)}u\\
		&=\displaystyle\sum_{u\in(\np(P)\setminus\npo(P))\cap I_2}\frac{\z(V(P,u))}{h_P(u)}u
		+\displaystyle\sum_{u\in(\np(Q)\setminus\npo(Q))\cap I_2}\frac{\z(V(Q,u))}{h_Q(u)}u.
	\end{align*}
	
	Fix $u\in S^{n-1}$. Since for $P\in\poly$, $P\mapsto a_{P}(u)$ is a valuation, we have
	\[a_{P\cup Q}(u)+a_{P\cap Q}(u)=a_{P}(u)+a_{Q}(u)\]
	for all $P,Q\in\poly$ with $P\cup Q\in\poly$. Note that
	\[h_{P\cup Q}(u)=h_{P\cap Q}(u)=h_P(u)=h_Q(u)\]
	for $u\in I_{2}$. Then,
	\[V(P\cup Q,u)+V(P\cap Q,u)=V(P,u)+V(Q,u)\]
	for $u\in I_2$. Since $\z$ is a solution of Cauchy's functional equation, we obtain
	\begin{equation}\label{eqn:3.4}
	\frac{\z(V(P\cup Q,u))}{h_{P\cup Q}(u)}+\frac{\z(V(P\cap Q,u))}{h_{P\cap Q}(u)}
	=\frac{\z(V(P,u))}{h_P(u)}+\frac{\z(V(Q,u))}{h_Q(u)}
	\end{equation}
	for $u\in I_2$, where $P,Q\in\poly$ with $P\cup Q\in\poly$. Also, note that
	\begin{equation}\label{eqn:3.5}
		\np(P\cup Q)\cup\np(P\cap Q)=\np(P)\cup\np(Q).
	\end{equation}
	Combined with \eqref{eqn:3.4} and \eqref{eqn:3.5}, we obtain the desired valuation property.
	
	Next, we will show that the facet vector operator vanishes in the following two cases.
	
	If $\dim P\leq n-2$, it is clear that $\fv_\z(P)=0$ as $\np(P)=\varnothing$.
	
	If $\dim P=n-1$, then $h_P(u)=-h_P(-u)$, where $u,-u$ are the outer unit normals of $P$.
By the definition of the facet vector, we obtain $\fv_\z(P)=0$.
\end{proof}

Next, we prove the $\SL$ contravariance of the facet vector.

\begin{lem}\label{thm:contra}
	Let $\z:[0,\infty)\to\R$ be a solution of Cauchy's functional equation. Then,
	the facet vector operator $\fv_\z:\poly\to\R^n$ is $\SL$ contravariant.
\end{lem}
\begin{proof}
	Let $\f\in\SL$. Note that
	\begin{equation}\label{eqn:3.6}
		u\in\np(P)\setminus\npo(P)~~\Leftrightarrow~~\tilde{u}\in\np(\f P)\setminus\npo(\f P)
	\end{equation}
	with
	\[\tilde{u}:=\norm{\f^{-t}u}^{-1}\f^{-t}u\]
	and that
	\[h_{\f P}(\tilde{u})=h_{P}(\f^{t}\tilde{u})=\norm{\f^{-t}u}^{-1}h_P(u),\,
	a_{\f P}(\tilde{u})=\norm{\phi^{-t}u}a_P(u).\]
	We have
	\begin{equation}\label{eqn:3.7}
		V(\f P,\tilde{u})=V(P,u).
	\end{equation}
	Applying \eqref{eqn:3.6}, \eqref{eqn:3.7} and the definition of the facet vector, we obtain
	\[\begin{array}{rl}
	\displaystyle\fv_\z(\f P)
	&=\displaystyle\sum_{\tilde{u}\in\np(\f P)\setminus\npo(\f P)}\frac{\z(V(\f P,\tilde{u}))}{h_{\f P}(\tilde{u})}\tilde{u}\\
	&=\displaystyle\sum_{u\in\np(P)\setminus\npo(P)}\frac{\z(V(P,u))}{\norm{\f^{-t}u}^{-1}h_P(u)}
	(\norm{\f^{-t}u}^{-1}\f^{-t}u)\\
	&=\displaystyle\sum_{u\in\np(P)\setminus\npo(P)}\frac{\z(V(P,u))}{h_P(u)}\f^{-t}u\\
	&=\displaystyle\f^{-t}\fv_\z(P).
	\end{array}\]
	Thus, we have finished the proof of the $\SL$ contravariance of the facet vector.
\end{proof}

Finally, the facet vector is related to an $\SLe$ covariant valuation
in dimension two up to a rotation.
Let $\z:[0,\infty)\to\R$ be a solution of Cauchy's functional equation.
Define $\vo{\z}:\polyeo\to\R^2$ by
\[\vo{\z}(P)=\sum_{i=2}^r\frac{\z\left(\det(v_{i-1},v_i)\right)}{\det(v_{i-1},v_i)}(v_{i-1}-v_i)\]
if $\dim P=2$ and
$P=\conv{o,v_1,\ldots,v_r}$ with $o\in\bd P$ and the vertices $\{o,v_1,\ldots, v_r\}$ are labeled counter-clockwisely;
\[\vo{\z}(P)=\frac{\z\left(\det(v_r, v_1)\right)}{\det(v_r, v_1)}(v_r-v_1) +\sum_{i=2}^r\frac{\z\left(\det(v_{i-1},v_i)\right)}{\det(v_{i-1},v_i)}(v_{i-1}-v_i)\]
if $o\in\intp P$ and $P=\conv{v_1,\ldots,v_r}$
with the vertices $\{v_1,\ldots, v_r\}$ are labeled counter-clockwisely;
\[\vo{\z}(P)=0\]
if $P=\set{o}$ or $P$ is a line segment.

\begin{lem}\label{thm:hfv}
	Let $\z:[0,\infty)\to\R$ be a solution of Cauchy's functional equation. Then
	\[\fv_\z(P)=\frac{1}{2}\rho_{\frac{\pi}{2}}\vo{\z}(P),\]
	for all $P\in\polyeo$.
\end{lem}
\begin{proof}
	For $\dim P=2$ and $P=\conv{o,v_1,\ldots,v_r}$ with $o\in\bd P$ and
	the vertices $\{0,v_1,\ldots, v_r\}$ are labeled counter-clockwisely,
	we have
	\[\vo{\z}(P)=\sum_{i=2}^r\frac{\z\left(\det(v_{i-1},v_i)\right)}{\det(v_{i-1},v_i)}(v_{i-1}-v_i)
	=\sum_{i=2}^r\frac{\z\left(2V(\conv{o,v_{i-1},v_i})\right)}{2V(\conv{o,v_{i-1},v_i})}(v_{i-1}-v_i).\]
	Write $u_i=\frac{\rho_{\frac{\pi}{2}}(v_{i-1}-v_i)}{\norm{v_{i-1}-v_i}}$.
	Then, $u_i$ is the outer unit normal of $\conv{v_{i-1},v_i}$ and
	$\conv{o,v_{i-1},v_i}$ is the cone $\conv{o,F(P,u_i)}$. Therefore,
	\begin{align*}
	\rho_{\frac{\pi}{2}}\vo{\z}(P)
	&=\sum_{i=2}^r\frac{\z\left(2V(P,u_i)\right)}{2V(P,u_i)}\norm{v_{i-1}-v_i}u_i\\
	&=2\sum_{i=2}^r\frac{\z\left(V(P,u_i)\right)}{\norm{v_{i-1}-v_i}h_P(u_i)}\norm{v_{i-1}-v_i}u_i\\
	&=2\sum_{u\in\np(P)\setminus\npo(P)}\frac{\z(V(P,u))}{h_P(u)}u\\
	&=2\fv_\z(P).
	\end{align*}
	Similar arguments also prove other cases.
\end{proof}


\section{Proof of the main results on $\polyo$}

\subsection{The two-dimensional case}

First, we show a relation between $\SLe$ covariant functions
and $\SLe$ contravariant functions.
Let $\mathcal Q^2$ be either $\polyeo$ or $\polye$.

\begin{lem}\label{thm:cocontra}
	Let $ Z :\mathcal Q^2\to\R^2$. Then, $ Z $ is $\SLe$ covariant if and only if
	$\rho_{\frac{\pi}{2}} Z $ is $\SLe$ contravariant.
\end{lem}
\begin{proof}
    A direct calculation shows that
	\begin{equation}\label{eqn:-trho}
		\rho_{\frac{\pi}{2}}\f=\f^{-t}\rho_{\frac{\pi}{2}},
	\end{equation}
	for all $\f\in\SLe$.
    First, we assume that $ Z $ is $\SLe$ covariant.
    Together with \eqref{eqn:-trho}, we have
	\[\rho_{\frac{\pi}{2}} Z (\f P)=(\rho_{\frac{\pi}{2}}\f) Z (P)
	=(\f^{-t}\rho_{\frac{\pi}{2}}) Z (P)=\f^{-t}(\rho_{\frac{\pi}{2}} Z (P))\]
    for all $P\in\mathcal Q^2$ and $\f\in\SLe$.
	This proves that $\rho_{\frac{\pi}{2}} Z $ is $\SLe$ contravariant.
	
	Next, we assume $\rho_{\frac{\pi}{2}} Z $ is contravariant. Together with \eqref{eqn:-trho}, we have
	\[\rho_{\frac{\pi}{2}} Z (\f P)=(\f^{-t}\rho_{\frac{\pi}{2}}) Z (P)
	=(\rho_{\frac{\pi}{2}}\f) Z (P)=\rho_{\frac{\pi}{2}}(\f Z (P)).\]
	Hence, $ Z (\f P)=\f Z (P)$, which completes the proof.
\end{proof}

We will use the following result.

\begin{thm}[\cite{ZM18}]\label{thm:co2o}
	A function $ Z :\polyeo\to\R^2$ is an $\SLe$ covariant valuation
	if and only if there exist constants $c_1, c_2\in\R$
	and a solution of Cauchy's functional equation $\z:[0,\infty)\to\R$ such that
	\[ Z (P)=c_1\mv(P)+c_2\ve(P)+\vo{\z}(P)\]
	for every $P\in\polyeo$.
\end{thm}

Using the relation and theorem above, we obtain the following proof.

\begin{proof}[Proof of Theorem \ref{thm:main20}]
	By Lemma \ref{thm:cocontra}, $ Z $ is an $\SLe$ contravariant valuation
	if and only if $\rho_{\frac{\pi}{2}}^{-1} Z $ is an $\SLe$ covariant valuation.
	Then, by Theorem \ref{thm:co2o}, there exist constants $c_1, c_2\in\R$
	and a solution of Cauchy's functional equation $\bar{\z}:[0,\infty)\to\R$ such that
	\[\rho_{\frac{\pi}{2}}^{-1} Z (P)=\vo{\bar{\z}}(P)+c_1\mv(P)+c_2\ve(P)\]
	i.e.
	\[ Z (P)=\rho_{\frac{\pi}{2}}\vo{\bar{\z}}(P)+c_1\rho_{\frac{\pi}{2}}\mv(P)+c_2\rho_{\frac{\pi}{2}}\ve(P)\]
	for every $P\in\polyeo$. Set $\z=2\bar{\z}$. By Lemma \ref{thm:hfv}, we obtain
	\[ Z (P)=\fv_\z(P)+c_1\rho_{\frac{\pi}{2}}\mv(P)+c_2\rho_{\frac{\pi}{2}}\ve(P),\]
	which completes the proof.
\end{proof}


\subsection{The higher-dimensional case}

First, we state the following simple proposition.

\begin{prop}\label{thm:id}
	Let $n\geq3$ and $ Z :\polyo\to\R^n$ be an $\SL$ contravariant function.
	Then, there exists a constant $a\in\R$ such that $ Z (T^n)=a\id$.
\end{prop}
\begin{proof}
	We first consider $n=3$. Write $ Z (T^3)=(x_1,x_2,x_3)^t$ and
	\[\sg_0=\left(\begin{array}{ccc}
		0 & 0 & 1\\
		1 & 0 & 0\\
		0 & 1 & 0
	\end{array}\right)\in\SLs.\]
	The $\SLs$ contravariance of $ Z $ implies
	\[ Z (T^3)= Z (\sg_0T^3)=\sg_0^{-t} Z (T^3),\]
	i.e.
	\[\left(\begin{array}{c}x_1\\x_2\\x_3\end{array}\right)=\left(\begin{array}{ccc}
		0 & 0 & 1\\
		1 & 0 & 0\\
		0 & 1 & 0
	\end{array}\right)\left(\begin{array}{c}x_1\\x_2\\x_3\end{array}\right)=\left(\begin{array}{c}x_3\\x_1\\x_2\end{array}\right).\]
	Thus, $x_1=x_2=x_3$.
	
	Next, we consider $n\geq4$. Write $ Z (T^n)=(x_1,\ldots,x_n)^t$ and
	\[\sg=\left(\begin{array}{ccc}
		I_r & &\\
		& \sg_0 &\\
		& & I_{n-r-3}
	\end{array}\right)\in\SL,\]
	where $r=0,1,\ldots,n-3$ and $\sg_0$ moves along the main diagonal of $\sg$.
	Using the $\SL$ contravariance of $ Z $, we have $ Z (T^n)= Z (\sg T^n)=\sg^{-t} Z (T^n)$. This yields $x_1=\cdots=x_n$. Therefore, $ Z (T^n)=a\id$ for some $a \in \R$.
\end{proof}

Next, we obtain a property of simple valuations.

\begin{lem}\label{thm:stk}
	Let $n\geq2$ and $ Z :\polyo\to\R^n$ be an $\SL$ contravariant valuation.
	Then, $ Z $ is simple if $ Z (T^k)=0$ for $k=0,1,\ldots,n-1$.
\end{lem}
\begin{proof}
    First, using triangulations of polytopes,
	it suffices to prove $ Z $ vanishes on $\mathcal T^k$ for $k=0,1,\ldots,n-1$.
	Since every $T\in\mathcal T^k$ is an $\SL$ image of $sT^k$ for $s\neq0$, we only need to consider $sT^k$. Now, write
	\[\rho=\left(\begin{array}{ccc}
		sI_k & &\\
		& I_{n-k-1} &\\
		& & s^{-k}
	\end{array}\right)\in\SL.\]
	The $\SL$ contravariance of $ Z $ gives $ Z (sT^k)= Z (\rho T^k)=\rho^{-t} Z (T^k)$.
	By the assumption that $ Z (T^k)=0$ for $k=0,1,\ldots,n-1$, we obtain that $ Z $ vanishes on all $sT^k$ for $s\neq0$ and $k=0,1,\ldots,n-1$. Therefore, $ Z $ is simple.
\end{proof}

Now, we investigate $\SL$ contravariant valuations on $\mathcal T^k$.

\begin{lem}\label{thm:nsimple}
	Let $n\geq3$ and $ Z :\polyo\to\R^n$ be an $\SL$ contravariant valuation.
	Then, $ Z $ is simple.
\end{lem}
\begin{proof}
	Due to Lemma \ref{thm:stk}, it suffices to prove $ Z $ vanishes on $T^k$ for $k=0,1,\ldots,n-1$.
	We prove the statement by induction on the dimension $k$.
	
	For $k=0$, write $ Z (\set{o})=(v_1,\ldots,v_n)^t$,
	\[\sg_1=\left(\begin{array}{cc}
		-1 & 0\\
		0 & -1
	\end{array}\right)\text{ and }\sg_2=\left(\begin{array}{ccc}
		I_r & &\\
		& \sg_1 &\\
		& & I_{n-r-2}
	\end{array}\right)\in\SL,\]
	where $r=0,1,\ldots,n-2$.
	Using the $\SL$ contravariance of $ Z $, we have $ Z (\set{o})= Z (\sg_2\set{o})=\sg_2^{-t} Z (\set{o})$.
	Hence, $v_1=\cdots=v_n=0$.
	
	For $k=1$, write $ Z (T^1)=(w_1,\ldots,w_n)^t$ and
	\[\sg_3=\begin{pmatrix}
		I_r & &\\
		& \sg_1 &\\
		& & I_{n-r-2}
	\end{pmatrix}\in\SL,\]
	where $r=1,\ldots,n-2$.
	Using the $\SL$ contravariance of $ Z $, we have $ Z (T^1)= Z (\sg_3T^1)=\sg_3^{-t} Z (T^1)$.
	Thus, $w_2=\cdots=w_n=0$ and $ Z (T^1)=w_1e_1$.
	
	For $k=2$, write $ Z (T^2)=(x_1,\ldots,x_n)^t$.
	If $n=3$, we consider
	\[\sg_4=\begin{pmatrix}
		0 & 1 & 0\\
		1 & 0 & 0\\
		0 & 0 & -1
	\end{pmatrix}\in\SLs.\]
	The $\SLs$ contravariance of $ Z $ implies $ Z (T^2)= Z (\sg_4T^2)=\sg_4^{-t} Z (T^2)$.
	Thus, $x_1=x_2$ and $x_3=0$.
	If $n\geq4$, we consider
	\[\sg_5=\begin{pmatrix}
		\sg_4 & 0\\
		0 & I_{n-3}
	\end{pmatrix}\in\SL\text{ and }\sg_6=\begin{pmatrix}
		I_r & &\\
		& \sg_1 &\\
		& & I_{n-r-2}
	\end{pmatrix}\in\SL,\]
	where $r=2,\ldots,n-2$. By the $\SL$ contravariance of $ Z $,
	we have $ Z (T^2)= Z (\sg_5T^2)=\sg_5^{-t} Z (T^2)$ and $ Z (T^2)= Z (\sg_6T^2)=\sg_6^{-t} Z (T^2)$.
	Thus, $x_1=x_2$, $x_3=\cdots=x_n=0$ and $ Z (T^2)=x_1(e_1+e_2)$.
	Now, we use the triangulation in Definition \ref{defn:fp1}.
	Equation \eqref{eqn:extk} is equivalent to
	\[\begin{pmatrix}
		\frac{1}{\l} & -\frac{1-\l}{\l} & 0 & \cdots & 0\\
		-\frac{\l}{1-\l} & \frac{1}{1-\l} & 0 & \cdots & 0\\
		0 & 0 & 1 & \cdots & 0\\
		\vdots & \vdots & \vdots & \ddots & \vdots\\
		0 & 0 & 0 & \cdots & 0
	\end{pmatrix}\begin{pmatrix}x_1\\x_2\\0\\ \vdots\\0\end{pmatrix}
	=\begin{pmatrix}
		\frac{1}{\l} & -\frac{1-\l}{\l} & 0 & \cdots & 0\\
		0 & 1 & 0 & \cdots & 0\\
		0 & 0 & 1 & \cdots & 0\\
		\vdots & \vdots & \vdots & \ddots & \vdots\\
		0 & 0 & 0 & \cdots & \l
	\end{pmatrix}\begin{pmatrix}w_1\\0\\0\\ \vdots\\0\end{pmatrix}.\]
	This yields $x_1=w_1=0$. Therefore, $ Z $ vanishes on $T^1$ and $T^2$.

    Next, assume $ Z (T^{k-1})=0$ for $3\leq k\leq n-1$.
    Write $ Z (T^k)=(y_1,\ldots,y_n)^t$ and
    \[\sg_7=\begin{pmatrix}
		 &1 &  &\\
        1&  &  &\\
		 &  &I_{n-3} &\\
		 &  &  & -1
	\end{pmatrix}\in\SL.\]
    By the $\SL$ contravariance of $ Z $,
	we have $ Z (T^k)= Z (\sg_7T^k)=\sg_7^{-t} Z (T^k)$.
    Thus $y_1=y_2$.
	
	Finally, we use the triangulation in Definition \ref{defn:fp1}. Equation \eqref{eqn:extk} is equivalent to
	\[\begin{pmatrix}
		\frac{1}{\l} & -\frac{1-\l}{\l} & 0 & \cdots & 0\\
		-\frac{\l}{1-\l} & \frac{1}{1-\l} & 0 & \cdots & 0\\
		0 & 0 & 1 & \cdots & 0\\
		\vdots & \vdots & \vdots & \ddots & \vdots\\
		0 & 0 & 0 & \cdots & 1
	\end{pmatrix}\begin{pmatrix}y_1\\y_2\\y_3\\ \vdots\\y_n\end{pmatrix}
	=\begin{pmatrix}0\\0\\0\\ \vdots\\0\end{pmatrix}.\]
    Together with $y_1=y_2$, this yields $y_1=\cdots=y_n=0$.
    Therefore, $ Z (T^k)=0$, which completes the proof.
\end{proof}

Finally, we obtain the following classification.

\begin{proof}[Proof of Theorem \ref{thm:mainN0}]
	Let $\z:[0,\infty)\to\R$ be a solution of Cauchy's functional equation.
	Due to Lemmas \ref{thm:simpval} and \ref{thm:contra},
	$\fv_\z$ is an $\SL$ contravariant valuation on $\polyo$. It remains to show the reverse statement.
	
	We use the triangulation in Definition \ref{defn:fp2}. By \eqref{eqn:extn} and Lemma \ref{thm:nsimple},
	we have for $s>0$
	\[ Z (s^{\frac{1}{n}}T^n) =\l^{\frac{1}{n}}\f_2^{-t} Z ((\l s)^{\frac{1}{n}}T^n) +(1-\l)^{\frac{1}{n}}\p_2^{-t} Z (((1-\l)s)^{\frac{1}{n}}T^n).\]
	By Proposition \ref{thm:id}, there exists a function $f:[0,\infty)\to\R$ such that $ Z (sT^n)=f(s)\id$ and
	\[f(s^{\frac{1}{n}}) \id =\l^{\frac{1}{n}}\f_2^{-t} f((\l s)^{\frac{1}{n}})\id +(1-\l)^{\frac{1}{n}}\p_2^{-t} f(((1-\l)s)^{\frac{1}{n}}) \id.\]
	In other words,
	\[f(s^{\frac{1}{n}})=\l^{\frac{1}{n}}f((\l s)^{\frac{1}{n}}) +(1-\l)^{\frac{1}{n}}f(((1-\l)s)^{\frac{1}{n}}).\]
	Set $s=a+b$, $\l=a/(a+b)$ for $a,b>0$, and $g(x)=x^{\frac{1}{n}}f(x^{\frac{1}{n}})$ for $x>0$ to get
	\[g(a+b)=g(a)+g(b).\]
	Hence, $g$ is a solution of Cauchy's functional equation and
	\[ Z (s^{\frac{1}{n}}T^n)=\frac{g(s)}{s^{\frac{1}{n}}}\id.\]
	Setting $\z\left(s/n!\right)=g(s)$, we obtain $ Z (s^{\frac{1}{n}}T^n)=\fv_\z(s^{\frac{1}{n}}T^n)$.
	The proof is now completed by Lemma \ref{lemuq}.
\end{proof}


\section{Proof of the main results on $\poly$}

First, we treat the case for $\polye$.
We need the following result.

\begin{thm}[\cite{ZM18}]\label{thm:co2}
	A function $ Z :\polye\to\R^2$ is an $\SLe$ covariant valuation
	if and only if there exist constants $c_1,c_2,\tilde c_1,\tilde c_2\in\R$
	and solutions of Cauchy's functional equation $\z_1,\z_2:[0,\infty)\to\R$ such that
	\begin{align*}
    Z (P)&=\vo{\z_1}(\conv{o,P})+\sum_{i=2}^r\vo{\z_2}(\conv{o,v_{i-1},v_i}) \\
    &\qquad +c_1\mv(P)+\tilde c_1\mv(\conv{o,P})+c_2\ve(\conv{o,P})
	+\tilde c_2\ve(\conv{o,v_1,\ldots,v_r})
    \end{align*}
	for every polytope $P\in\polye$ with vertices $v_1,\ldots,v_r$ visible from the origin and labeled counter-clockwisely.
\end{thm}

Now, similar to the proof of Theorem \ref{thm:main20},
we obtain the characterization in dimension two.

\begin{proof}[Proof of Theorem \ref{thm:main2}]
	By Lemma \ref{thm:cocontra}, $ Z $ is an $\SLe$ contravariant valuation
	if and only if $\rho_{\frac{\pi}{2}}^{-1} Z $ is an $\SLe$ covariant valuation.
	Then, by Theorem \ref{thm:co2}, there exist constants $c_1,c_2,\tilde c_1,\tilde c_2\in\R$
	and solutions of Cauchy's functional equation $\bar\z_1,\bar\z_2:[0,\infty)\to\R$ such that
	\begin{align*}
		\rho_{\frac{\pi}{2}}^{-1} Z (P)&=\vo{\bar\z_1}(\conv{o,P})+\sum_{i=2}^r\vo{\bar\z_2}(\conv{o,v_{i-1},v_i})\\
		&\qquad +c_1\mv(P)+\tilde c_1\mv (\conv{o,P})+c_2\ve(\conv{o,P})+\tilde c_2\ve(\conv{o,v_1,\ldots,v_r})
	\end{align*}
	i.e.
	\begin{align*}
		Z (P)&=\rho_{\frac{\pi}{2}}\vo{\bar\z_1}(\conv{o,P})+\sum_{i=2}^r\rho_{\frac{\pi}{2}}\vo{\bar\z_2}(\conv{o,v_{i-1},v_i})\\
		&\qquad +c_1\rho_{\frac{\pi}{2}}\mv(P)+\tilde c_1\rho_{\frac{\pi}{2}}\mv(\conv{o,P})
		+c_2\rho_{\frac{\pi}{2}}\ve(\conv{o,P})+\tilde c_2\rho_{\frac{\pi}{2}}\ve(\conv{o,v_1,\ldots,v_r})
	\end{align*}
	for every polytope $P\in\polye$ with vertices $v_1,\ldots,v_r$ visible from the origin and labeled counter-clockwisely.
	By Lemma \ref{thm:hfv}, we obtain
	\begin{align*}
		 Z (P)&=\fv_{2\bar\z_1}(\conv{o,P})+\sum_{i=2}^r\fv_{2\bar\z_2}(\conv{o,v_{i-1},v_i})\\
		&\qquad +c_1\rho_{\frac{\pi}{2}}\mv(P)+\tilde c_1\rho_{\frac{\pi}{2}}\mv(\conv{o,P})
		+c_2\rho_{\frac{\pi}{2}}\ve(\conv{o,P})+\tilde c_2\rho_{\frac{\pi}{2}}\ve(\conv{o,v_1,\ldots,v_r}).
	\end{align*}
	Furthermore, Lemma \ref{thm:simpval} yields
	\[\sum_{i=2}^r\fv_{2\bar\z_2}(\conv{o,v_{i-1},v_i})=\fv_{2\bar\z_2}(\conv{o,P})-\fv_{2\bar\z_2}(P).\]
	Finally, we set $\zeta_1=-2\bar\z_2$ and $\z_2=2(\bar\z_1+\bar\z_2)$ to conclude the proof.
\end{proof}

In the final step, we extend Theorem \ref{thm:mainN0} to $\poly$.

\begin{proof}[Proof of Theorem \ref{thm:mainN}]
	Let $\z_1,\z_2:[0,\infty)$ be solutions of Cauchy's functional equation.
	First, due to Lemmas \ref{thm:simpval} and \ref{thm:contra},
	$\fv_{\z_1}$ is an $\SL$ contravariant valuation on $\poly$.
	Next, for $P,Q\in\poly$ with $P\cup Q\in\poly$, we have
	$\conv{o,P\cup Q}=\conv{o,P}\cup\conv{o,Q}$ and $\conv{o,P\cap Q}=\conv{o,P}\cap\conv{o,Q}$.
	Notice that $\conv{o,\f P}=\f\conv{o,P}$ for all $\f\in\SL$ and $P\in\poly$.
	Again by Lemmas \ref{thm:simpval} and \ref{thm:contra}, we obtain that
	the function $P\mapsto\fv_{\z_2}(\conv{o,P})$ for $P\in\poly$
	is also an $\SL$ contravariant valuation on $\poly$.
		
	It remains to show the reverse statement. Indeed, we only need to show
	that $ Z $ has the corresponding representation on $sT^k$ and $s\tilde T^k$ for $s>0$ and $0\leq k\leq n$.
	By Theorem \ref{thm:mainN0}, there exists a solution of Cauchy's functional
	equation $\eta_1:[0,\infty)$ such that
	\[ Z (sT^k)=\fv_{\eta_1}(sT^k).\]
	
	Let $\spl$ be the set of simplices in $\R^n$ with one vertex at the origin.
	For any $T\in\splo$, we write $\tilde T$ as its facet opposite to the origin.
	We define the new map $\tilde Z :\spl\to\R$ by $\tilde Z(T)= Z \tilde(T)$
	for every $T\in\splo$ and $Z\{o\}=o$. It is not hard to check that $\tilde Z $ is
	an $\SL$ contravariant valuation on $\spl$. From the proof of Theorem \ref{thm:mainN0},
	one can see that Theorem \ref{thm:mainN0} also holds on $\spl$.
	Hence there exists a solution of Cauchy's functional equation $\eta_2:[0,\infty)$ such that
	\[ Z (s\tilde T^k)=\tilde Z (sT^k)=\fv_{\eta_2}(sT^k).\]
	
	Now, we set $\z_1=\eta_1-\eta_2$ and $\z_2=\eta_2$ such that
	\eqref{eqn:mainN} holds for both $sT^k$ and $s\tilde T^k$ for $0\leq k\leq n$,
	which completes the proof by Lemma \ref{lemuq2}.
\end{proof}


\section*{Acknowledgement}
\addcontentsline{toc}{section}{Acknowledgement}

The work of the first author was supported in part by the Austrian Science Fund (FWF M2642 and I3027)
and by the National Natural Science Foundation of China (Project 11671249).
The work of the second author was supported in part by
the National Natural Science Foundation of China (Project 11701373)
and by Shanghai Sailing Program 17YF1413800.
The second author is the corresponding author.


\end{document}